\documentclass[10pt]{amsart}

\usepackage{times,amsfonts,amsmath,amstext,amsbsy,amssymb,
  amsopn,amsthm,upref,eucal, amscd}
\usepackage[T1]{fontenc}
\usepackage{color}

\usepackage[pdftex]{graphicx}

\newtheorem{theorem}{Theorem}[section]
\newtheorem{lemma}[theorem]{Lemma}
\newtheorem{corollary}[theorem]{Corollary}
\newtheorem{proposition}[theorem]{Proposition}

\numberwithin{equation}{section}

\theoremstyle{definition}

\newtheorem{remark}[theorem]{Remark}

\def\leq{\leqslant }
\def\geq{\geqslant}

\begin{document}

\title[$HD(M\setminus L) > 0.353$]{$HD(M\setminus L) > 0.353$}

\author[C. Matheus and C. G. Moreira]{Carlos Matheus and Carlos Gustavo Moreira}

\address{Carlos Matheus:
Universit\'e Paris 13, Sorbonne Paris Cit\'e, CNRS (UMR 7539),
F-93430, Villetaneuse, France.
}

\email{matheus.cmss@gmail.com}

\address{Carlos Gustavo Moreira:
IMPA, Estrada Dona Castorina 110, 22460-320, Rio de Janeiro, Brazil
}

\email{gugu@impa.br}

\date{\today}

\begin{abstract}
The complement $M\setminus L$ of the Lagrange spectrum $L$ in the Markov spectrum $M$ was studied by many authors (including Freiman, Berstein, Cusick and Flahive). After their works, we disposed of a countable collection of points in $M\setminus L$.

In this article, we describe the structure of $M\setminus L$ near a non-isolated point $\alpha_{\infty}$ found by Freiman in 1973, and we use this description to exhibit a concrete Cantor set $X$ whose Hausdorff dimension coincides with the Hausdorff dimension of $M\setminus L$ near $\alpha_{\infty}$. 

A consequence of our results is the lower bound $HD(M\setminus L)>0.353$ on the Hausdorff dimension $HD(M\setminus L)$ of $M\setminus L$. Another by-product of our analysis is the explicit construction of new elements of $M\setminus L$, including its largest known member $c\in M\setminus L$ (surpassing the former largest known number $\alpha_4\in M\setminus L$ obtained by Cusick and Flahive in 1989).
\end{abstract}
\maketitle


\section{Introduction}

\subsection{Statement of the main results} The Lagrange and Markov spectra are subsets of the real line related to classical Diophantine approximation problems. More precisely, the \emph{Lagrange spectrum} is
$$L:=\left\{\limsup\limits_{\substack{p, q\to\infty \\ p, q\in\mathbb{Z}}} \frac{1}{|q(q\alpha-p)|}<\infty:\alpha\in\mathbb{R}-\mathbb{Q}\right\}$$
and the \emph{Markov spectrum} is
$$M:=\left\{\frac{1}{\inf\limits_{\substack{(x,y)\in\mathbb{Z}^2\\(x,y)\neq(0,0)}} |q(x,y)|}<\infty: q(x,y)=ax^2+bxy+cy^2 \textrm{ real indefinite, }  b^2-4ac=1 \right\}.$$

Markov proved in 1879 that
$$L\cap (-\infty, 3) = M\cap (-\infty, 3) = \left\{\sqrt{5}<\sqrt{8}<\frac{\sqrt{221}}{5}<\dots\right\}$$ consists of an \emph{explicit} increasing sequence of quadratic surds accumulating only at $3$.

Hall proved in 1947 that $L\cap [c,\infty) = [c,\infty)$ for some constant $c>3$. For this reason, a half-line $[c,\infty)$ contained in the Lagrange spectrum is called a \emph{Hall ray}.

Freiman determined in 1975 the biggest half-line $[c_F,\infty)$ contained in the Lagrange spectrum, namely,
$$c_F:=\frac{2221564096+283748\sqrt{462}}{491993569} \simeq 4.5278\dots.$$
The constant $c_F$ is called \emph{Freiman's constant}.

In general, it is known that $L\subset M$ are closed subsets of $\mathbb{R}$. The results of Markov, Hall and Freiman mentioned above imply that the Lagrange and Markov spectra coincide below $3$ and above $c_F$. Nevertheless, Freiman showed in 1968 that $M\setminus L\neq\emptyset$ by exhibiting a number $\sigma\simeq 3.1181\dots\in M\setminus L$. On the other hand, some authors believe that the Lagrange and Markov spectra coincide\footnote{\emph{Added in proof}: This was conjectured by Cusick in 1975: see \cite[p. 516]{C}. As it turns out, after this article was completed, we discovered that Cusick's conjecture is false, namely $M\setminus L$ contains some numbers near $3.7096\dots$: see \cite{MM} for more details.} above $\sqrt{12}\simeq 3.4641\dots$.

The reader is invited to consult the excellent book \cite{CF} of Cusick-Flahive for a review of the literature on the Lagrange and Markov spectrum until the mid-eighties.

The main theorem of this paper concerns the Hausdorff dimension of $M\setminus L$:

\begin{theorem}\label{t.A} The Hausdorff dimension $HD(M\setminus L)$ of $M\setminus L$ satisfies:
$$0.353 < HD(M\setminus L)$$
\end{theorem}

The proof of Theorem \ref{t.A} is based on a refinement of the analysis in Chapter 3 of Cusick-Flahive book \cite{CF} of a sequence $\alpha_n\in M\setminus L$, $n\geq 4$, converging to a number $\alpha_{\infty}\simeq 3.293\dots\in M\setminus L$ in order to exhibit a Cantor set $X$ such that
$$0.353<HD(X)=HD((M\setminus L)\cap (b_{\infty}, B_{\infty})),$$
where $(b_{\infty}, B_{\infty})$ is the largest interval disjoint from $L$ containing $\alpha_{\infty}$.

\begin{remark}\label{r.Bumby} The Cantor set $X$ is described in \eqref{e.Cantor-X} below: it is a Cantor set defined in terms of \emph{explicit} restrictions on continued fraction expansions. In particular, one can use the ``thermodynamical arguments'' of Bumby \cite{B}, Hensley \cite{H}, Jenkinson-Pollicott \cite{JP16} and Falk-Nussbaum \cite{FN} to compute $HD(X)$.

In this direction, we implemented the algorithm of Jenkinson-Pollicott and we obtained the \emph{heuristic} approximation $HD(X)=0.4816\cdots$.

In principle, this heuristic approximation can be made rigorous, but we have not pursued this direction. Instead, we exhibit a Cantor set $K(\{1,2_2\})\subset X$ whose Hausdorff dimension can be easily (and rigorously) estimated as $0.353<HD(K(\{1,2_2\}))<0.35792$ via some classical arguments explained in Palis-Takens book \cite{PT}: see Section \ref{a.PT} below.
\end{remark}

As it turns out, the first term $\alpha_4 = 3.29304427\dots$ of the sequence $(\alpha_n)_{n\geq 4}$ mentioned above was the largest \emph{known} element of $M\setminus L$ since 1989 (see page 35 of \cite{CF}). By exploiting the arguments establishing Theorem \ref{t.A}, we are able to exhibit new numbers in $M\setminus L$, including a constant $c\in M\setminus L$ with $c>\alpha_4$:

\begin{proposition}\label{p.new-numbers} The largest element of $(M\setminus L)\cap (b_{\infty}, B_{\infty})$ is
$$c = \frac{77+\sqrt{18229}}{82}+\frac{17633692-\sqrt{151905}}{24923467}=3.29304447990138\dots.$$
In particular, $c$ is the largest known element of $M\setminus L$.
\end{proposition}

\begin{remark} One has $\frac{c-\alpha_{\infty}}{\alpha_4-\alpha_{\infty}} = 32.58\dots$. In other words, if the coordinates are centered at $\alpha_{\infty}$, then $c$ is more than $32$ times larger than $\alpha_4$.
\end{remark}

\subsection{Organization of the article} In Section \ref{s.preliminaries}, we recall some classical facts about continued fractions and Perron's characterization of $L$ and $M$. In Section \ref{s.HD(M-L)>0}, we show that $HD((M\setminus L)\cap (b_{\infty}, B_{\infty})) = HD(X)$, where $X$ is a Cantor set of real numbers in $[0,1]$ whose continued fraction expansions correspond to the elements of $\{1,2\}^{\mathbb{N}}$ not containing nine explicit finite words. In particular, this reduces the proof of Theorem \ref{t.A} to the computation of lower bounds on $HD(X)$. In Section \ref{a.PT}, we complete the proof of Theorem \ref{t.A} by showing that $HD(K(\{1,2_2\}))>0.353$, where $K(\{1,2_2\})\subset X$ is the Cantor set of real numbers in $[0,1]$ whose continued fraction expansions associated to elements of $\{1,2\}^{\mathbb{N}}$ given by concatenations of the finite words $1$ and $2,2$. In Section \ref{a.new-numbers}, we pursue the arguments in Section \ref{s.HD(M-L)>0} in order to establish Proposition \ref{p.new-numbers}. In Appendix \ref{a.Berstein}, we show that $(b_{\infty}, B_{\infty})$ is the largest interval disjoint from $L$ containing $\alpha_{\infty}$: in particular, we correct some claims made by Berstein in Theorem 1 at page 47 of \cite{Be73} concerning $(b_{\infty}, B_{\infty})$. Finally, in Appendix \ref{a.CF}, we show that the largest element $\alpha_2$ of the sequence $(\alpha_n)_{n\in\mathbb{N}}$ constructed by Cusick-Flahive in Chapter 3 of \cite{CF} belongs to the Lagrange spectrum. 

\subsection*{Acknowledgements} We are thankful to Thomas Cusick, Dmitry Gayfulin and Nikolay Moshchevitin for their immense help in giving us access to the references \cite{Be73} and \cite{Fr73}. 

\section{Some preliminaries}\label{s.preliminaries}

\subsection{Continued fractions} Given an irrational number $\alpha$, we denote by
$$\alpha=[a_0; a_1, a_2,\dots] = a_0+\frac{1}{a_1+\frac{1}{a_2+\frac{1}{\ddots}}}$$
its continued fraction expansion,
and we let
$$[a_0; a_1,\dots, a_n] := a_0+\frac{1}{a_1+\frac{1}{\ddots+\frac{1}{a_n}}} := [a_0; a_1,\dots, a_n,\infty,\dots]$$  be its $n$th convergent.

A standard comparison tool for continued fractions is the following lemma\footnote{Compare with Lemmas 1 and 2 in Chapter 1 of Cusick-Flahive book \cite{CF}.}:

\begin{lemma}\label{l.0} Let $\alpha=[a_0; a_1,\dots, a_n, a_{n+1},\dots]$ and $\beta=[a_0; a_1,\dots, a_n, b_{n+1},\dots]$ with $a_{n+1}\neq b_{n+1}$. Then:
\begin{itemize}
\item $\alpha>\beta$ if and only if $(-1)^{n+1}(a_{n+1}-b_{n+1})>0$;
\item $|\alpha-\beta|<1/2^{n-1}$.
\end{itemize}
\end{lemma}

\begin{remark}\label{r.0} For later use, note that Lemma \ref{l.0} implies that if $a_0\in\mathbb{Z}$ and $a_i\in\mathbb{N}^*$ for all $i\geq 1$, then $[a_0; a_1,\dots, a_n,\dots]<[a_0; a_1,\dots, a_n]$ when $n\geq 1$ is odd, and $[a_0; a_1,\dots, a_n,\dots]>[a_0; a_1,\dots, a_n]$ when $n\geq 0$ is even.
\end{remark}

\subsection{Perron's description of the Lagrange and Markov spectra} Given a bi-infinite sequence $A=(a_n)_{n\in\mathbb{Z}}\in(\mathbb{N}^*)^{\mathbb{Z}}$ and $i\in\mathbb{Z}$, let
$$\lambda_i(A) := [a_i; a_{i+1}, a_{i+2}, \dots] + [0; a_{i-1}, a_{i-2}, \dots].$$

Define the quantities
$$\ell(A)=\limsup\limits_{i\to\infty}\lambda_i(A) \quad \textrm{and} \quad m(A) = \sup\limits_{i\in\mathbb{Z}} \lambda_i(A).$$

In 1921, Perron showed that
$$L=\{\ell(A)<\infty: A\in(\mathbb{N}^*)^{\mathbb{Z}}\} \quad \textrm{and} \quad M=\{m(A)<\infty: A\in(\mathbb{N}^*)^{\mathbb{Z}} \}.$$

In the sequel, we will work exclusively with these characterizations of $L$ and $M$.

\subsection{Gauss-Cantor sets} Given a finite alphabet $B=\{\beta_1,\dots,\beta_m\}$, $m\geq 2$, consisting of finite words $\beta_j\in(\mathbb{N}^*)^{r_j}$, $1\leq j\leq m$, such that $\beta_i$ does not begin by $\beta_j$ for all $i\neq j$, we denote by
$$K(B):=\{[0;\gamma_1, \gamma_2,\dots]: \gamma_i\in B \,\,\,\, \forall \, i\geq 1\}\subset [0,1]$$
the \emph{Gauss-Cantor set} associated to $B$.

\subsection{Some notations} Given a finite word $\beta=(b_1,\dots, b_r)\in(\mathbb{N}^*)^r$, we denote by $\beta^T:=(b_r,\dots,b_1)$ the \emph{transpose} of $\beta$.

Also, we abreviate periodic continued fractions and bi-infinite sequences which are periodic in one or both sides by putting a bar over the period: for instance, $[\overline{2,1,1}] = [2;1,1,2,1,1,2,1,1,\dots]$ and $\overline{1}, 2, 1, 2, \overline{1, 2} = \dots, 1, 1, 1, 2, 1, 2, 1, 2, 1, 2, 1, 2, \dots$.

Moreover, we shall use subscripts to indicate the multiplicity of a digit in a sequence: for example, $[2; 1_2, 2_3, 1, 2, \dots] = [2; 1, 1, 2, 2, 2, 1, 2, \dots]$.

\section{$HD(M\setminus L)>0$}\label{s.HD(M-L)>0}

In 1973, Freiman \cite{Fr73} showed that
$$\alpha_{\infty}:=\lambda_0(A_{\infty}):=[2; \overline{1_2, 2_3, 1, 2}] + [0; 1, 2_3, 1_2, 2, 1, \overline{2}]\in M \setminus L.$$
In a similar vein, Theorem 4 in Chapter 3 of Cusick-Flahive book \cite{CF} asserts that
$$\alpha_n:=\lambda_0(A_n):= [2; \overline{1_2, 2_3, 1, 2}] + [0; 1, 2_3, 1_2, 2, 1, 2_n, \overline{1, 2, 1_2, 2_3}]\in M\setminus L$$
for all $n\geq 4$. In particular, $\alpha_{\infty}$ is not isolated in $M\setminus L$.

In what follows, we shall revisit Freiman's arguments as described in Chapter 3 of Cusick-Flahive book \cite{CF} in order to prove the following result. Let $X$ be the Cantor set
\begin{equation}\label{e.Cantor-X}
X:=\{[0;\gamma]:\gamma\in\{1,2\}^{\mathbb{N}} \textrm{ not containing the subwords in } P\}
\end{equation}
where
$$P:=\{21212, 2121_3, 1_3212, 12121_2, 1_22121, 2_3121_22_21, 12_21_2212_3, 12_3121_22_2,
2_21_2212_31\}$$
Also, let
$$b_{\infty} := [2;\overline{1_2,2_3,1,2}]+[0;\overline{1,2_3,1_2,2}] = 3.2930442439\dots$$
and
\begin{eqnarray*}
B_{\infty} &:=& [2;1,\overline{1,2_3,1,2,1_2,2,1_2,2}]+[0;1,2_3,1_2,2,1,2_3,1_2,2,1,2_2,\overline{1,2_3,1,2,1_2,2,1_2,2}] \\
&=& 3.2930444814\dots.
\end{eqnarray*}

The remainder of this section is devoted to the proof of the following result:

\begin{theorem}\label{t.M-L-piece-HD} $HD((M\setminus L)\cap (b_{\infty}, B_{\infty})) = HD(X)$ (where $X$ is the Cantor set in \eqref{e.Cantor-X}).
\end{theorem}

\subsection{Description of $M\setminus L$ near $\alpha_{\infty}$}\label{a.M-L-near-alphainfty}

Our description of $(M\setminus L)\cap (b_{\infty}, B_{\infty})$
needs the following improvements of Lemma 1 in \cite[Chapter 3]{CF}:

\begin{lemma}\label{l.1} If $B\in\{1,2\}^{\mathbb{Z}}$ contains any of the subsequences:
\begin{itemize}
\item[(i)] $212^*12$
\item[(ii)] $212^*1_3$
\item[(iii)] $1212^*1_2$
\item[(iv)] $2_3 1 2^* 1_2 2_2 1$
\item[(v)] $2 1 2_3 1 2^* 1_2 2_3$
\item[(vi)] $1_2 2_3 1 2^* 1_2 2_4$
\item[(vii)] $1_2 2_3 1 2^* 1_2 2_3 1_2$
\item[(viii)] $1_3 2_3 1 2^* 1_2 2_3 1 2$
\item[(ix)] $2 1_2 2_3 1 2^* 1_2 2_3 1 2_2$
\item[(x)] $2_2 1_2 2_3 1 2^* 1_2 2_3 1 2 1$
\item[(xi)] $1_2 2 1_2 2_3 1 2^* 1_2 2_3 1 2 1_2 2$
\end{itemize}
then $\lambda_j(B)>\alpha_{\infty}+10^{-6}$ where $j$ indicates the position in asterisk.
\end{lemma}

\begin{proof} If (i) occurs, then Remark \ref{r.0} implies that
$$\lambda_j(B) = [2; 1, 2,\dots]+[0; 1, 2,\dots] > [2; 1, 2]+[0; 1, 2] = \frac{10}{3} > \alpha_{\infty} + 10^{-2}.$$

If (ii) occurs, then Remark \ref{r.0} says that
$$\lambda_j(B) = [2; 1_3,\dots] + [0; 1, 2,\dots] > [2; 1_4] + [0; 1, 2_2, 1] = \frac{33}{10} > \alpha_{\infty} + 10^{-3}.$$

If (iii) occurs, then Remark \ref{r.0} implies that
\begin{eqnarray*}
\lambda_j(B) &=& [2; 1_2,\dots] + [0; 1,2,1,\dots] \\ 
&>& [2; 1_2,2,1,2,1] + [0; 1, 2, 1_2,2,1] = \frac{2143}{650} > \alpha_{\infty}+10^{-3}.
\end{eqnarray*}

If (iv) occurs, then Remark \ref{r.0} says that
\begin{eqnarray*}
\lambda_j(B) &=& [2; 1_2, 2_2, 1, \dots] + [0; 1, 2_3,\dots]  \\
&>& [2; 1_2, 2_2, 1_2, 2, 1] + [0; 1, 2_4, 1] = \frac{9933}{3016} > \alpha_{\infty}+10^{-4}.
\end{eqnarray*}

If (v) occurs, then Remark \ref{r.0} implies that
\begin{eqnarray*}
\lambda_j(B) &=& [2; 1_2, 2_3,\dots] + [0; 1, 2_3, 1, 2,\dots]  \\
&>& [2; 1_2, 2_3, 1] + [0; 1, 2_3, 1, 2] = \frac{8776}{2665} > \alpha_{\infty}+10^{-5}.
\end{eqnarray*}

If (vi) occurs, then Remark \ref{r.0} says that
\begin{eqnarray*}
\lambda_j(B) &=& [2; 1_2, 2_4, \dots] + [0; 1, 2_3, 1_2, \dots] \\
&>& [2; 1_2, 2_5, 1] + [0; 1, 2_3, 1_2, 2, 1] = \frac{115702}{35133} > \alpha_{\infty}+10^{-4}.
\end{eqnarray*}

If (vii) occurs, then Remark \ref{r.0} says that
\begin{eqnarray*}
\lambda_j(B) &=& [2; 1_2, 2_3, 1_2, \dots] + [0; 1, 2_3, 1_2, \dots] \\
&>& [2; 1_2, 2_3, 1_3, 2, 1] + [0; 1, 2_3, 1_2, 2, 1] = \frac{195086}{59241} > \alpha_{\infty}+10^{-5}.
\end{eqnarray*}

If (viii) occurs, then Remark \ref{r.0} implies that
\begin{eqnarray*}
\lambda_j(B) &=& [2; 1_2, 2_3, 1, 2,\dots] + [0; 1, 2_3, 1_3, \dots] \\
&>& [2; 1_2, 2_3, 1, 2, 1] + [0; 1, 2_3, 1_4] = \frac{26529}{8056} > \alpha_{\infty}+10^{-5}.
\end{eqnarray*}

If (ix) occurs, then Remark \ref{r.0} says that
\begin{eqnarray*}
\lambda_j(B) &=& [2; 1_2, 2_3, 1, 2_2, \dots] + [0; 1, 2_3, 1_2, 2, \dots] \\
&>& [2; 1_2, 2_3, 1, 2_3, 1] + [0; 1, 2_3, 1_2, 2, 1, 2, 1] = \frac{1621169}{492300} > \alpha_{\infty}+10^{-6}.
\end{eqnarray*}

If (x) occurs, then Remark \ref{r.0} implies that
\begin{eqnarray*}
\lambda_j(B) &=& [2; 1_2, 2_3, 1, 2, 1, \dots] + [0; 1, 2_3, 1_2, 2_2, \dots] \\
&>& [2; 1_2, 2_3, 1, 2, 1, 2, 1] + [0; 1, 2_3, 1_2, 2_3, 1] = \frac{1615094}{490455} > \alpha_{\infty} + 10^{-6}.
\end{eqnarray*}

If (xi) occurs, then Remark \ref{r.0} says that
\begin{eqnarray*}
\lambda_j(B) &=& [2; 1_2, 2_3, 1, 2, 1_2, 2,\dots] + [0; 1, 2_3, 1_2, 2, 1_2, \dots] \\
&>& [2; 1_2, 2_3, 1, 2, 1_2, 2] + [0; 1, 2_3, 1_2, 2, 1_2, 1] = \frac{446537}{135600} > \alpha_{\infty} + 10^{-6}.
\end{eqnarray*}
%
\end{proof}

\begin{lemma}\label{l.xii} Let $B\in\{1,2\}^{\mathbb{Z}}$.
\begin{itemize}
\item[(xii')] If $B$ contains $1 2 1 2_2 1 2 1_2 2_3 1 2^* 1_2 2_3 1 2 1_2 2_3 1 2 1_2 2$, then $\lambda_j(B) > B_{\infty} + 6\times 10^{-9}$ where $j$ indicates the position in asterisk.
\item[(xii'')] If $B$ contains $2_2 1 2_2 1 2 1_2 2_3 1 2^* 1_2 2_3 1 2 1_2 2_3 1 2 1_2 2$, then $\lambda_j(B) < B_{\infty} - 10^{-9}$ where $j$ indicates the position in asterisk.
\end{itemize}
\end{lemma}

\begin{proof} If (xii') occurs, then Lemma \ref{l.0} says that
\begin{eqnarray*}
\lambda_j(B) &=& [2; 1_2, 2_3, 1, 2, 1_2, 2_3, 1, 2, 1_2, 2,\dots] + [0; 1, 2_3, 1_2, 2, 1, 2_2, 1, 2, 1, \dots] \\
&\geq& [2; 1_2, 2_3, 1, 2, 1_2, 2_3, 1, 2, 1_2, 2,\overline{1,2}] + [0; 1, 2_3, 1_2, 2, 1, 2_2, 1, 2, 1, \overline{1,2}] \\
&>& B_{\infty} + 6\times 10^{-9}.
\end{eqnarray*}

If (xii'') occurs, then Lemma \ref{l.0} says that
\begin{eqnarray*}
\lambda_j(B) &=& [2; 1_2, 2_3, 1, 2, 1_2, 2_3, 1, 2, 1_2, 2,\dots] + [0; 1, 2_3, 1_2, 2, 1, 2_2, 1, 2_2 \dots] \\
&\leq& [2; 1_2, 2_3, 1, 2, 1_2, 2_3, 1, 2, 1_2, 2,\overline{2,1}] + [0; 1, 2_3, 1_2, 2, 1, 2_2, 1, 2_2, \overline{2,1}] \\
&<& B_{\infty} - 10^{-9}.
\end{eqnarray*}
\end{proof}

We will also need the following result (extracted from Lemma 2 in Chapter 3 of \cite{CF}):

\begin{lemma}\label{l.2} If $B\in\{1,2\}^{\mathbb{Z}}$ contains any of the subsequences
\begin{itemize}
\item[(a)] $1^*$
\item[(b)] $22^*$
\item[(c)] $1_2 2^* 1_2$
\item[(d)] $2_2 1 2^* 1_2 2 1$
\item[(e)] $1 2_2 1 2^* 1_2 2$
\item[(f)] $2_4 1 2^* 1_2 2_3$
\end{itemize}
then $\lambda_j(B)<\alpha_{\infty} - 10^{-5}$ where $j$ indicates the position in asterisk.
\end{lemma}

\begin{proof} If (a) occurs, then $\lambda_j(B) = 1+[0;\dots]+[0;\dots]<3<\alpha_{\infty}-10^{-1}$.

If (b) occurs, then Remark \ref{r.0} implies that
$$\lambda_j(B) = [2;2,\dots]+[0;\dots]<[2; 1, 2, 1] + [0; 2, 2, 1] = \frac{89}{28} < \alpha_{\infty}-10^{-1}.$$

If (c) occurs, then Remark \ref{r.0} says that
$$\lambda_j(B) = [2; 1, 1, \dots]+[0; 1, 1, \dots] < [2; 1_3, 2, 1] + [0; 1_3, 2, 1] = \frac{36}{11} < \alpha_{\infty}-10^{-2}.$$

If (d) occurs, then Remark \ref{r.0} implies that
\begin{eqnarray*}
\lambda_j(B) &=& [2; 1_2, 2, 1, \dots]+[0; 1, 2_2, \dots] \\
&<& [2; 1_2, 2, 1_2, 2, 1] + [0; 1, 2_3, 1] = \frac{3395}{1032} < \alpha_{\infty}-10^{-3}.
\end{eqnarray*}

If (e) occurs, then Remark \ref{r.0} says that
\begin{eqnarray*}
\lambda_j(B) &=& [2; 1_2, 2, \dots]+[0; 1, 2_2, 1, \dots] \\
&<& [2; 1_2, 2_2, 1, 2, 1] + [0; 1, 2_2, 1_2, 2, 1, 2, 1] = \frac{47081}{14301} < \alpha_{\infty} - 10^{-4}.
\end{eqnarray*}

If (f) occurs, then Remark \ref{r.0} implies that
\begin{eqnarray*}
\lambda_j(B) &=& [2; 1_2, 2_3, \dots]+[0; 1, 2_4, \dots] \\
&<& [2; 1_2, 2_4, 1] + [0; 1, 2_5, 1] = \frac{45641}{13860} < \alpha_{\infty} - 10^{-5}.
\end{eqnarray*}
\end{proof}

By putting together Lemma \ref{l.1} and Lemma \ref{l.2}, we obtain:

\begin{lemma}\label{l.first-restriction} Let $B=(B_m)_{m\in\mathbb{Z}}\in\{1,2\}^{\mathbb{Z}}$ be a bi-infinite sequence. Suppose that $\lambda_n(B)\leq\alpha_{\infty}+10^{-6}$ at a certain position $n\in\mathbb{Z}$. Then, the sole possible situations are:
\begin{itemize}
\item $B_n=1$ and $\lambda_n(B)<\alpha_{\infty}-10^{-5}$;
\item $B_{n-1}B_n=22$ and $\lambda_n(B)<\alpha_{\infty}-10^{-5}$;
\item $B_nB_{n+1}=22$ and $\lambda_n(B)<\alpha_{\infty}-10^{-5}$;
\item $B_{n-2}B_{n-1}B_nB_{n+1}B_{n+2} = 11211$ and $\lambda_n(B)<\alpha_{\infty}-10^{-5}$;
\item $B_{n-3}\dots B_{n+4}\in \{21121221, 22121121\}$ and $\lambda_n(B)<\alpha_{\infty}-10^{-5}$;
\item $B_{n-4}\dots B_{n+3}\in \{12112122, 12212112\}$ and $\lambda_n(B)<\alpha_{\infty}-10^{-5}$;
\item $B_{n-5}\dots B_{n+5}\in\{22211212222, 22221211222\}$ and $\lambda_n(B)<\alpha_{\infty}-10^{-5}$;
\item $B_{n-5}\dots B_{n+4}=1222121122$;
\item $B_{n-4}\dots B_{n+5}= 2211212221$.
\end{itemize}
In particular, the subwords $212^*12$, $212^*1_3$, $1_32^*12$, $1212^*1_2$, $1_22^*121$,  $2_312^*1_22_21$ and $12_21_22^*12_3$ are forbidden (where the asterisk indicates the $n$th position).
\end{lemma}

\begin{proof} By items (a) and (b) of Lemma \ref{l.2}, if $B_n=1$, $B_{n-1}B_n=22$ or $B_n B_{n+1}=22$, then $\lambda_n(B)<\alpha_{\infty}-10^{-5}$.

By items (i), (ii) and (iii) of Lemma \ref{l.1}, our assumption $\lambda_n(B)\leq\alpha_{\infty}+10^{-6}$ implies that the subwords $212^*12$, $212^*1_3$, $1_32^*12$, $1212^*1_2$ and $1_22^*121$ are forbidden for $B_n=2^*$. So, if $B_{n-1}B_nB_{n+1}=121$, then one has just three possibilities:
\begin{itemize}
\item $B_{n-2}\dots B_{n+2}=11211$ and, by item (c) of Lemma 2, $\lambda_n(B)<\alpha_{\infty}-10^{-5}$;
\item $B_{n-3}\dots B_{n+3}\in \{2112122, 2212112\}$.
\end{itemize}

Suppose that $B_{n-3}\dots B_{n+3}\in \{2112122, 2212112\}$. By items (d) and (e) of Lemma \ref{l.2}, if $B_{n+4}=1$ or $B_{n-4}=1$, i.e., if
$$B_{n-3}\dots B_{n+4}\in \{21121221, 22121121\} \quad \textrm{or} \quad
B_{n-4}\dots B_{n+3}\in \{12112122, 12212112\},$$
then $\lambda_n(B)<\alpha_{\infty}-10^{-5}$.

Assume that $B_{n-4}\dots B_{n+4}\in\{221121222, 222121122\}$. By item (f) of Lemma \ref{l.2}, if $(B_{n-5}, B_{n+5})=(2,2)$, i.e.,
$$B_{n-5}\dots B_{n+5}\in\{22211212222, 22221211222\},$$
then $\lambda_n(B)<\alpha_{\infty}-10^{-5}$.

Consider the case $B_{n-4}\dots B_{n+4}\in\{221121222, 222121122\}$ and $(B_{n-5}, B_{n+5})\neq (2,2)$. Our assumption $\lambda_n(B)\leq\alpha_{\infty}+10^{-6}$ and the item (iv) of Lemma \ref{l.1} say that the subwords $22212^*11221$ and $122112^*1222$ are forbidden for $B_n=2^*$. Therefore, we have just two possibilities in this situation:
$$B_{n-5}\dots B_{n+4}=1222121122 \quad \textrm{or} \quad B_{n-4}\dots B_{n+5}= 2211212221.$$

This proves the desired lemma.
\end{proof}

By further exploiting Lemma \ref{l.1}, we also get the following results:

\begin{lemma}\label{l.second-restriction} Let $B=(B_m)_{m\in\mathbb{Z}}\in\{1,2\}^{\mathbb{Z}}$ be a bi-infinite sequence. Suppose that $\lambda_{n}(B) \leq \alpha_{\infty}+10^{-6}$ for some $n\in\mathbb{Z}$.
\begin{itemize}
\item If $B_{n-5}\dots B_{n+4}=1222121122$, then $B_{n-8}\dots B_{n+8}=12112221211222121$;
\item If $B_{n-4}\dots B_{n+5}= 2211212221$, then $B_{n-8}\dots B_{n+8} = 12122211212221121$.
\end{itemize}
\end{lemma}

\begin{proof} Since $2211212221=(1222121122)^T$, it suffices to show the lemma in the first case $B_{n-5}\dots B_{n+4}=1222121122$.

By succesively using items (iv), (v), (vi), (vii), (viii), (ix) and (x) of Lemma \ref{l.1} together with our assumption $\lambda_n(B)\leq\alpha_{\infty}+10^{-6}$, we see that, in our setting, the only possible way to extend $B_{n-5}\dots B_{n+4}=1222121122$ is $B_{n-8}\dots B_{n+8}=1211222121122$.
\end{proof}

\begin{lemma}\label{l.third-restriction} Let $B=(B_m)_{m\in\mathbb{Z}}\in\{1,2\}^{\mathbb{Z}}$ be a bi-infinite sequence. Suppose that $\lambda_{n-7}(B), \lambda_{n}(B), \lambda_{n+7}(B) \leq \alpha_{\infty}+10^{-6}$ for some $n\in\mathbb{Z}$.
\begin{itemize}
\item If $B_{n-5}\dots B_{n+4}=12_3121_22_2$, then:
\begin{itemize}
\item either $B_{n-10}\dots B_{n+11}=2_2121_22_3121_22_3121_221$,
\item or $B_{n-10}\dots B_{n+11}=2_2121_22_3121_22_3121_22_2$ and, in particular, the vicinity of $B_{n+7}=2$ is $B_{n+2}\dots B_{n+11} = 12_3121_22_2$.
\end{itemize}
\item If $B_{n-4}\dots B_{n+5}= 2_21_2212_31$, then:
\begin{itemize}
\item either $B_{n-11}\dots B_{n+10} = 121_2212_31_2212_31_2212_2$,
\item or $B_{n-11}\dots B_{n+10} = 2_21_2212_31_2212_31_2212_2$ and, in particular, the vicinity of $B_{n-7}=2$ is $B_{n-11}\dots B_{n-2} = 2_21_2212_31$.
\end{itemize}
\end{itemize}
\end{lemma}

\begin{proof} Since $2211212221=(1222121122)^T$, it suffices to show the lemma in the first case $B_{n-5}\dots B_{n+4}=1222121122$.

By Lemma \ref{l.second-restriction}, we have from our hypothesis $\lambda_n(B)\leq \alpha_{\infty}+10^{-6}$ that $B_{n-8}\dots B_{n+8}=121_22_3121_2 2_3121$.

From our assumption $\lambda_{n+7}(B)\leq \alpha_{\infty}+10^{-6}$ and the items (i), (ii) of Lemma \ref{l.1}, the only way to extend $B_{n-8}\dots B_{n+8}$ is
$$B_{n-8}\dots B_{n+10} = 121_22_3121_2 2_3121_22.$$

From our assumption $\lambda_{n}(B)\leq \alpha_{\infty}+10^{-6}$ and the item (xi) of Lemma \ref{l.1}, the only way to extend $B_{n-8}\dots B_{n+10}$ is
$$B_{n-9}\dots B_{n+10} = 2121_22_3121_2 2_3121_22.$$

From our assumption $\lambda_{n-7}(B)\leq \alpha_{\infty}+10^{-6}$ and the item (iii) of Lemma \ref{l.1}, the only way to extend $B_{n-9}\dots B_{n+10}$ is
$$B_{n-10}\dots B_{n+10} = 2_2121_22_3121_2 2_3121_22.$$

Thus, $B_{n-10}\dots B_{n+10}$ extends as
$$B_{n-10}\dots B_{n+11} = 2_2121_22_3121_22_3121_221 \quad \textrm{ or } \quad 2_2121_22_3121_2 2_3121_22_2.$$
\end{proof}

Next, we employ Lemmas \ref{l.1} and \ref{l.xii} to get the following statement:

\begin{lemma}\label{l.Binfty} Let $A\in\{1,2\}^{\mathbb{Z}}$ be a bi-infinite sequence. Suppose that, for some $n\in\mathbb{Z}$ and $a\in\mathbb{N}$, one has $\lambda_{n\pm7}(A) \leq B_{\infty}+6\times 10^{-9}$, $\lambda_{n\pm(17+6k)}(A)\leq \alpha_{\infty}+10^{-6}$ for each $k=1,\dots, 2a$, and $\lambda_{n\pm(7+6j)}(A)\leq \alpha_{\infty}+10^{-6}$ for each $j=1,\dots, 2a$.

If $A_{n-10}\dots A_{n+11}$ or $(A_{n-11}\dots A_{n+10})^T$ equals $2_2121_22_3121_22_3121_221$, then
\begin{eqnarray*}
\lambda_n(A)&\geq& [2;1,\underbrace{1,2_3,1,2,1_2,2,1_2,2,\dots, 1,2_3,1,2,1_2,2,1_2,2}_{a+1 \textrm{ times }},\dots] \\
&+& [0;1,2_3,1_2,2,1,2_3,1_2,2,1,2_2,\underbrace{1,2_3,1,2,1_2,2,1_2,2,\dots, 1,2_3,1,2,1_2,2,1_2,2}_{a \textrm{ times }},\dots].
\end{eqnarray*}

In particular, the subsequence $2_2121_22_3121_22_3121_221$ or its transpose $121_2212_31_2212_31_2212_2$ is not contained in a bi-infinite sequence $A\in\{1,2\}^{\mathbb{Z}}$ with $m(A)<B_{\infty}$.
\end{lemma}

\begin{proof} We can assume that $A_{n-10}\dots A_{n+11}=2_2121_22_3121_22_3121_221$: indeed, the other case $(A_{n-11}\dots A_{n+10})^T = 2_2121_22_3121_22_3121_221$ is completely similar.

By Lemma \ref{l.0}, if $A_{n-10}\dots A_{n+11}=2_2121_22_3121_22_3121_221$, then
$$\lambda_n(A) \geq [2;1_2,2_3,1,2,1_2,2,1_2,2,1,2,\dots]+[0;1,2_3,1_2,2,1,2_3,1,\dots].$$

From our assumption $\lambda_{n-7}(A)\leq B_{\infty}+6\times 10^{-9}<\alpha_{\infty}+10^{-6}$, we deduce from the items (v), (viii), (x) and (xi) of Lemma \ref{l.1} that
$$\lambda_n(A) \geq [2;1_2,2_3,1,2,1_2,2,1_2,2,1,2,\dots]+[0;1,2_3,1_2,2,1,2_3,1_2,2,1,2,\dots].$$

By Lemma \ref{l.0}, one has
$$\lambda_n(A) \geq [2;1_2,2_3,1,2,1_2,2,1_2,2,1,2,\dots]+[0;1,2_3,1_2,2,1,2_3,1_2,2,1,2_2,1,2,\dots].$$

It follows from our assumption $\lambda_{n-7}(A)\leq B_{\infty}+6\times 10^{-9}$ and Lemma \ref{l.xii} that
$$\lambda_n(A) \geq [2;1_2,2_3,1,2,1_2,2,1_2,2,1,2,\dots]+[0;1,2_3,1_2,2,1,2_3,1_2,2,1,2_2,1,2_2,\dots].$$

By Lemma \ref{l.0}, we get that
$$\lambda_n(A) \geq [2;1_2,2_3,1,2,1_2,2,1_2,2,1,2,\dots]+[0;1,2_3,1_2,2,1,2_3,1_2,2,1,2_2,1,2_3,1,2,1,\dots].$$

We proceed now by induction. On one hand, by recursively using
\begin{itemize}
\item the item (iii) of Lemma \ref{l.1} and our assumption $\lambda_{n+1+12j}(A)\leq \alpha_{\infty}+10^{-6}$ for $j=1,\dots, a$,
\item Lemma \ref{l.0}, and
\item the items (i), (ii), (iv) and (v) of Lemma \ref{l.1} and our assumption $\lambda_{n+7+12j}(A)\leq \alpha_{\infty}+10^{-6}$ for $j=1,\dots, a$,
\end{itemize}
we derive that $\lambda_n(A)$ is minimized when $A_{n+4+12j}\dots A_{n+15+12j}=2_2121_221_2212$ for $j=1,\dots, a$. On the other hand, by recursively using
\begin{itemize}
\item the items (i), (ii), (iv) and (v) of Lemma \ref{l.1} our assumption $\lambda_{n-11-12k}(A)\leq \alpha_{\infty}+10^{-6}$ for $k=1,\dots, a$, and
\item Lemma \ref{l.0},
\item the item (iii) of Lemma \ref{l.1} and our assumption $\lambda_{n-17-12k}(A)\leq \alpha_{\infty}+10^{-6}$ for $k=1,\dots, a$,
\end{itemize}
we derive that $\lambda_n(A)$ is minimized when $A_{n-13-12k}\dots A_{n-24-12k}=121_2212_3121$ for $k=1,\dots, a$. Therefore,
\begin{eqnarray*}
\lambda_n(A)&\geq& [2;1,\underbrace{1,2_3,1,2,1_2,2,1_2,2,\dots, 1,2_3,1,2,1_2,2,1_2,2}_{a+1 \textrm{ times }}\dots] \\
&+& [0;1,2_3,1_2,2,1,2_3,1_2,2,1,2_2,\underbrace{1,2_3,1,2,1_2,2,1_2,2,\dots, 1,2_3,1,2,1_2,2,1_2,2}_{a \textrm{ times }},\dots].
\end{eqnarray*}

Finally, suppose that $A\in\{1,2\}^{\mathbb{Z}}$ is a bi-infinite sequence with $m(A)<B_{\infty}$ containing $2_2121_22_3121_22_3121_221$ or its transpose, say $A_{l-10}\dots A_{l+11}$ or $(A_{l-11}\dots A_{l+10})^T$ equals $2_2121_22_3121_22_3121_221$ for some $l\in\mathbb{Z}$. The previous discussion would then imply that
\begin{eqnarray*}
B_{\infty} &>& m(A) \geq \lambda_l(A) \\ &\geq& [2;1,\overline{1,2_3,1,2,1_2,2,1_2,2}]
+ [0;1,2_3,1_2,2,1,2_3,1_2,2,1,2_2,\overline{1,2_3,1,2,1_2,2,1_2,2}] \\
&:=& B_{\infty},
\end{eqnarray*}
a contradiction. This completes the proof of the lemma.
\end{proof}

At this point, we are ready to describe $M\cap (b_{\infty}, B_{\infty})$.

\begin{proposition}\label{p.Cusick-Flahive-thm4} If $\alpha\in M\cap (b_{\infty}, B_{\infty})$, then $\alpha\notin L$.
\end{proposition}

\begin{proof} Our argument is inspired by the proof of Theorem 4 in Chapter 3 of Cusick-Flahive book \cite{CF}.

Suppose that $\alpha\in L\cap (b_{\infty}, B_{\infty})$. Let  $B\in\{1,2\}^{\mathbb{Z}}$ be a bi-infinite sequence such that $\ell(B):=\limsup\limits_{i\to\infty}\lambda_i(B) = \alpha$.

Since $\alpha_{\infty}-10^{-5}<b_{\infty}< \alpha < B_{\infty}$, we can fix $N\in\mathbb{N}$ large enough such that
$$\lambda_n(B)<B_{\infty}$$
for all $|n|\geq N$, and we can select a monotone sequence $\{n_k\}_{k\in\mathbb{N}}$ such that $|n_k|\geq N$ and $\lambda_{n_k}(B)\geq\alpha_{\infty}-10^{-5}$ for all $k\in\mathbb{Z}$. Moreover, by reversing $B$ if necessary, we can assume that $n_k\to+\infty$ as $k\to\infty$ and $\limsup\limits_{n\to+\infty}\lambda_n(B) = \alpha$.

We have two possibilities:
\begin{itemize}
\item either the sequence $B_n B_{n+1}\dots$ contains the subsequence $2_2121_22_3121_22_3121_221$ or its transpose $121_2212_31_2212_31_2212_2$ for all
$n\geq N$,
\item or there exists $R\geq N$ such that $B_R B_{R+1}\dots$ does not contain the subsequence $2_2121_22_3121_22_3121_221$ or its transpose $121_2212_31_2212_31_2212_2$.
\end{itemize}

In the first scenario, let $\{m_k\}_{k\in\mathbb{N}}$ be a monotone sequence such that $B_{m_k-10}\dots B_{m_k+11}$ or $(B_{m_k-11}\dots B_{m_k+10})^T$ equals $2_2121_22_3121_22_3121_221$ for all $k\in\mathbb{N}$ and $m_k\to+\infty$ as $k\to\infty$. By Lemma \ref{l.Binfty}, the fact that $\lambda_n(B)<B_{\infty}$ for all $n\geq N$ would imply that
\begin{eqnarray*}
\lambda_{m_k}(B)&\geq& [2;1,\underbrace{1,2_3,1,2,1_2,2,1_2,2,\dots, 1,2_3,1,2,1_2,2,1_2,2}_{a_k+1 \textrm{ times }}\dots] \\
&+& [0;1,2_3,1_2,2,1,2_3,1_2,2,1,2_2,\underbrace{1,2_3,1,2,1_2,2,1_2,2,\dots, 1,2_3,1,2,1_2,2,1_2,2}_{a_k \textrm{ times }},\dots]
\end{eqnarray*}
where $a_k=\lfloor\frac{m_k-17-N}{6}\rfloor$. Since $a_k\to\infty$ as $k\to\infty$, it would follow that
$$B_{\infty} > \alpha\geq \limsup\limits_{k\to\infty} \lambda_{m_k}(B)\geq B_{\infty},$$
a contradiction.

In the second scenario, we note that, by Lemma \ref{l.first-restriction}, for each $k\in\mathbb{N}$, we have
\begin{itemize}
\item either $B_{n_k-5}\dots B_{n_k+4}=12_3121_22_2$,
\item or $B_{n_k-4}\dots B_{n_k+5} = 2_21_2212_31$.
\end{itemize}
If the first possibility occurs for some $k_0\in\mathbb{N}$ with $n_{k_0}\geq R+10$, then the facts that $\lambda_n(B)<B_{\infty}<\alpha_{\infty}+10^{-6}$ for all $n\geq N$ and the sequence $B_RB_{R+1}\dots$ does not contain the subsequence $2_2121_22_3121_22_3121_221$ allow to repeatedly apply Lemma \ref{l.third-restriction} at the positions $n_{k_0}+7a$, $a\in\mathbb{N}$, to deduce that the sequence $B$ has the form
$$\dots B_{n_{k_0}} B_{n_{k_0}+1} B_{n_{k_0}+2}\dots = \dots 2\overline{1_22_312}.$$
If the second possibility occurs for all $n_k> R+10$, then the facts that $\lambda_n(B)<B_{\infty}<\alpha_{\infty}+10^{-6}$ for all $n\geq N$ and the sequence $B_RB_{R+1}\dots$ does not contain the subsequence $121_2212_31_2212_31_2212_2$ allow to apply $d_k:=\lfloor\frac{n_k-4-R}{7}\rfloor$ times Lemma \ref{l.third-restriction} at the positions $n_k-7(j-1)$, $j=1,\dots, d_k$, to deduce that the sequence $B$ has the form
$$\dots B_{n_k-7d_k}\dots B_{n_k}\dots B_{n_k+10}\dots = \dots \underbrace{212_31_2,\dots, 212_31_2}_{d_k \textrm{ times}}212_31_2212_2\dots.$$
Because $R-4\leq n_k-7d_k\leq R+11$ and $n_k\to+\infty$, we deduce that $B$ has the form $\dots\overline{212_31_2}$.

In any case, the second scenario would imply that
$$b_{\infty} < \alpha=\limsup\limits_{n\to+\infty}\lambda_n(B) = \ell(\overline{1_22_312})=b_{\infty},$$
a contradiction.

In summary, the existence of $\alpha\in L\cap (b_{\infty}, B_{\infty})$ leads to a contradiction in any scenario. This proves the proposition.
\end{proof}

\begin{remark}\label{r.Berstein} As it was first observed in Theorem 1, pages 47 to 49 of Berstein's article \cite{Be73}, one can \emph{improve} Proposition \ref{p.Cusick-Flahive-thm4} by showing that $(b_{\infty}, B_{\infty})$ is the \emph{largest} interval disjoint from $L$ containing $\alpha_{\infty}$.

Actually, it does not take much more work to get this improved version of Proposition \ref{p.Cusick-Flahive-thm4}: in fact, since this proposition ensures that $L\cap(b_{\infty}, B_{\infty}) = \emptyset$, and we have that $b_{\infty}=\ell(\overline{1_22_312})\in L$, it suffices to prove that $B_{\infty}\in L$. For the sake of completeness (and also to correct some mistakes in \cite{Be73}), we show that $B_{\infty}\in L$ in Appendix \ref{a.Berstein} below.
\end{remark}

\begin{proposition}\label{p.Cantors-covering-M-L-piece} Let $m\in M\cap (b_{\infty}, B_{\infty})$. Then, $m=m(B)=\lambda_0(B)$ for a sequence $B\in\{1,2\}^{\mathbb{Z}}$ with the following properties:
\begin{itemize}
\item $B_{-10}\dots B_0 B_1\dots B_7\dots = 2_2121_22_312\overline{1_22_312}$;
\item there exists $N\geq 11$ such that $\dots B_{-N-1} B_{-N}$ is a word on $1$ and $2$ satisfying:
\begin{itemize}
\item it does not contain the subwords $21212$, $2121_3$, $1_3212$, $12121_2$, $1_22121$,  $2_3121_22_21$, $12_21_2212_3$ and $12_3121_22_2$,
\item if it contains the subword $2_21_2212_31=B_{n-4}\dots B_{n+5}$, then
$$\dots B_{n-7}\dots B_{n+10} = \overline{212_31_2}212_31_2212_2.$$
\end{itemize}
\end{itemize}
\end{proposition}

\begin{proof} Let $B\in\{1,2\}^{\mathbb{Z}}$ be a bi-infinite sequence such that $m=m(B)$. Since $m<B_{\infty}$, Proposition \ref{p.Cusick-Flahive-thm4} implies that $\limsup\limits_{i\to\infty}\lambda_i(B)=\ell(B)\leq b_{\infty}< m$.

Therefore, we can select $N_0$ large enough such that $\lambda_n(B)< \frac{b_{\infty}+m}{2} < m$ for all $|n|\geq N_0$. In particular, $m=m(B)=\lambda_{n_0}(B)$ for some $|n_0|<N_0$.

It follows that we can shift $B$ in order to obtain a sequence -- still denoted by $B$ -- such that $\lambda_0(B) = m(B) = m$. Since $m>b_{\infty}>\alpha_{\infty}-10^{-5}$, Lemma \ref{l.first-restriction} says that
$$B_{-5}\dots B_4 = 1222121122 \quad \textrm{or} \quad B_{-4}\dots B_5 = 2211212221.$$
Thus, by reversing $B$ if necessary, we obtain a bi-infinite sequence $B\in\{1,2\}^{\mathbb{Z}}$ such that $m=m(B)=\lambda_0(B)$ and $B_{-5}\dots B_4 = 1222121122$.

Because $\lambda_n(B)\leq m<B_{\infty}$ for all $n\in\mathbb{Z}$, we know from Lemma \ref{l.Binfty} that $B$ does not contain the subsequence $2_2121_22_3121_22_3121_221$, and, thus, we can successively apply Lemma \ref{l.third-restriction} at the positions $7k$, $k\in\mathbb{N}$, to get that
$$B_{-10}\dots B_0 B_1\dots B_7\dots = 2_2121_22_312\overline{1_22_312}.$$
Moreover, Lemma \ref{l.first-restriction} implies that the word $\dots B_{-11}$ does not contain the subwords $21212$, $2121_3$, $1_3212$, $12121_2$, $1_22121$,  $2_3121_22_21$ and $12_21_2212_3$.

Furthermore, the subword $12_3121_22_2$ can not appear in $\dots B_n$ for all $n\leq -11$. Indeed, if this happens, since $m(B)=m<B_{\infty}$, it would follow from Lemma \ref{l.Binfty} that $B$ does not contain the subsequence $2_2121_22_3121_22_3121_221$ and, hence, one could repeatedly apply Lemma \ref{l.third-restriction} to deduce that $B=\overline{1_22_312}$, a contradiction because this would mean that $b_{\infty}<m=m(B)=m(\overline{1_22_312})=b_{\infty}$.

In summary, we showed that there exists $N\geq 11$ such that the word $\dots B_{-N}$ does not contain the subwords $21212$, $2121_3$, $1_3212$, $12121_2$, $1_22121$,  $2_3121_22_21$, $12_21_2212_3$ and $12_3121_22_2$.

Finally, if the word $\dots B_{-11}$ contains the subword $2_21_2212_31=B_{n-4}\dots B_{n+5}$, since $B$ does not contain the subsequence $121_2212_31_2212_31_2212_2$ (thanks to Lemma \ref{l.Binfty} and the fact that $\lambda_n(B)<B_{\infty}$ for all $n\in\mathbb{Z}$), then one can apply Lemma \ref{l.third-restriction} at the positions $n-7k$ for all $k\in\mathbb{N}$ to get that
$$\dots B_{n-7}\dots B_{n+10} = \overline{212_31_2}212_31_2212_2.$$

This completes the proof of the proposition.
\end{proof}

\begin{remark} We use Proposition \ref{p.Cantors-covering-M-L-piece} to detect new numbers in $M\setminus L$: see Appendix \ref{a.new-numbers}.
\end{remark}

\subsection{Comparison between $M\setminus L$ near $\alpha_{\infty}$ and the Cantor set $X$}
The description of $(M\setminus L)\cap  (b_{\infty}, B_{\infty}) = M\cap  (b_{\infty}, B_{\infty})$ provided by Propositions \ref{p.Cusick-Flahive-thm4} and \ref{p.Cantors-covering-M-L-piece} allows us to compare this piece of $M\setminus L$ with the Cantor set
\begin{equation*}
X:=\{[0;\gamma]:\gamma\in\{1,2\}^{\mathbb{Z}} \textrm{ not containing the subwords in } P\}
\end{equation*}
where
$$P:=\{21212, 2121_3, 1_3212, 12121_2, 1_22121, 2_3121_22_21, 12_21_2212_3, 12_3121_22_2,
2_21_2212_31\}$$
introduced in \eqref{e.Cantor-X} above.

\begin{proposition}\label{p.M-L-piece-HD1} $(M\setminus L)\cap  (\alpha_{\infty}-10^{-8}, \alpha_{\infty}+10^{-8})$ contains the set
$$\{[2; \overline{1_2, 2_3, 1, 2}] + [0; 1, 2_3, 1_2, 2, 1, 2_4,\gamma]: 2_3\gamma\in \{1,2\}^{\mathbb{N}} \textrm{ does not contain the subwords in } P\}.$$
\end{proposition}

\begin{proof} Consider the sequence
$$B=\gamma^T,2_4, 1, 2, 1_2, 2_3, 1, 2; \overline{1_2, 2_3, 1, 2}$$ where $2_3\gamma\in \{1, 2\}^{\mathbb{N}}$ does not contain subwords in $P$ and $;$ serves to indicate the zeroth position.

On one hand, Remark \ref{r.0} implies that
$$\lambda_0(B)\leq [2; \overline{1_2, 2_3, 1, 2}] + [0; 1, 2_3, 1_2, 2, 1, 2_4, 1, 2, 1] < \alpha_{\infty}+10^{-8}$$
and
$$\lambda_0(B)\geq [2; \overline{1_2, 2_3, 1, 2}] + [0; 1, 2_3, 1_2, 2, 1, 2_4, 2, 1] > \alpha_{\infty} - 10^{-8},$$
and items (a), (b) and (f) of Lemma \ref{l.2} imply that
$$\lambda_n(B)<\alpha_{\infty}-10^{-5}$$ for all positions $n\geq-12$ except possibly for $n=7k$ with $k\geq 1$.

On the other hand,
\begin{eqnarray*}\lambda_{7k}(B) &=& [2; \overline{1_2, 2_3, 1, 2}] + [0;\underbrace{1, 2_3, 1_2, 2, \dots, 1, 2_3, 1_2, 2}_{k \textrm{ times }}, 1, 2_3, 1_2, 2, 1, 2_4,\dots] \\ &<& [2; \overline{1_2, 2_3, 1, 2}] + [0; 1, 2_3, 1_2, 2, 1, 2_3, 1_2, 2, 1] \\
&<& [2; \overline{1_2, 2_3, 1, 2}] + [0; 1, 2_3, 1_2, 2, 1, 2_4] \\
&<& [2; \overline{1_2, 2_3, 1, 2}] + [0; 1, 2_3, 1_2, 2, 1, 2_4,\dots] = \lambda_0(B),
\end{eqnarray*}
so that $\lambda_0(B)-\lambda_{7k}(B)>[0; 1, 2_3, 1_2, 2, 1, 2_4] - [0; 1, 2_3, 1_2, 2, 1, 2_3, 1_2, 2, 1]> 10^{-9}$ for all $k\geq 1$.

Moreover, since $2_3\gamma$ does not contain subwords in $P$, it follows from (the proof of) Lemma \ref{l.first-restriction} that $\lambda_n(B)<\alpha_{\infty}-10^{-5}$ for all $n\leq -13$.

This shows that $m(B)=\lambda_0(B)=[2; \overline{1_2, 2_3, 1, 2}] + [0; 1, 2_3, 1_2, 2, 1, 2_4,\gamma]$ belongs to $(M\setminus L)\cap  (\alpha_{\infty}-10^{-8}, \alpha_{\infty}+10^{-8})$.
\end{proof}

\begin{proposition}\label{p.M-L-piece-HD2} $(M\setminus L)\cap  (b_{\infty}, B_{\infty})$ is contained in the union of
$$\mathcal{C} = \{[2; \overline{1_2, 2_3, 1, 2}] + [0; 1, 2_3, 1_2, 2, 1, 2_2,\theta,\overline{1_2,2_3,1,2}]: \theta \textrm{ is a finite word in } 1 \textrm{ and } 2\}$$
and the sets
$$\mathcal{D}(\delta) = \{[2; \overline{1_2, 2_3, 1, 2}] + [0; 1, 2_3, 1_2, 2, 1, 2_2,\delta,\gamma]: \textrm{ no subword of }\gamma\in \{1,2\}^{\mathbb{N}} \textrm{ belongs to } P\},$$
where $\delta$ is a finite word in $1$ and $2$.
\end{proposition}

\begin{proof} By Proposition \ref{p.Cantors-covering-M-L-piece}, if $m\in (M\setminus L)\cap (b_{\infty}, B_{\infty})$, then $m=m(B)=\lambda_0(B)$ with
$$B=\gamma^T\delta^T2_2121_22_312^*\overline{1_22_312}$$
where the asterisk indicates the zeroth position, $\delta$ is a finite word in $1$ and $2$, and the infinite word $\gamma$ satisfies:
\begin{itemize}
\item $\gamma^T$ does not contain the subwords $21212$, $2121_3$, $1_3212$, $12121_2$, $1_22121$,  $2_3121_22_21$, $12_21_2212_3$ and $12_3121_22_2$,
\item if $\gamma^T$ contains the subword $2_21_2212_31$, then $\gamma^T=\overline{212_31_2}\mu^T$ with $\mu$ a finite word in $1$ and $2$.
\end{itemize}

It follows that:
\begin{itemize}
\item if $\gamma^T$ contains $2_21_2212_31$, then
$$m(B)=\lambda_0(B)=[2; \overline{1_2, 2_3, 1, 2}] + [0; 1, 2_3, 1_2, 2, 1, 2_2,\delta,\mu,\overline{1_2,2_3,1,2}]$$
where $\theta=\delta\mu$ is a finite word in $1$ and $2$, i.e., $m(B)\in\mathcal{C}$;
\item otherwise,
$$m(B) = \lambda_0(B)=[2; \overline{1_2, 2_3, 1, 2}] + [0; 1, 2_3, 1_2, 2, 1, 2_2,\delta,\gamma]$$
where $\gamma$ does not contain the subwords $21212$, $2121_3$, $1_3212$, $12121_2$, $1_22121$,  $2_3121_22_21$, $12_21_2212_3$, $12_3121_22_2$ and $2_21_2212_31$, i.e., $m(B)\in\mathcal{D}(\delta)$.
\end{itemize}
This completes the proof of the proposition.
\end{proof}

\subsection{Proof of Theorem \ref{t.M-L-piece-HD}}
By putting together Propositions \ref{p.M-L-piece-HD1} and \ref{p.M-L-piece-HD2}, we can derive Theorem \ref{t.M-L-piece-HD}.

Indeed, by Proposition \ref{p.M-L-piece-HD1}, $(M\setminus L)\cap  (b_{\infty}, B_{\infty})$ contains a set diffeomorphic to $X$ and, hence,
$$HD((M\setminus L)\cap  (b_{\infty}, B_{\infty})) \geq HD(X).$$

By Proposition \ref{p.M-L-piece-HD2}, $(M\setminus L)\cap   (b_{\infty}, B_{\infty})$ is contained in
$$\mathcal{C}\cup\bigcup\limits_{n\in\mathbb{N}} \left(\bigcup\limits_{\delta\in\{1,2\}^n}\mathcal{D}(\delta)\right).$$
Since $\mathcal{C}$ is a countable set and $\{\mathcal{D}(\delta):\delta\in\{1,2\}^n, n\in\mathbb{N}\}$ is a countable family of subsets diffeomorphic to $X$, it follows that
$$HD((M\setminus L)\cap (b_{\infty}, B_{\infty})) \leq HD(X).$$

This proves Theorem \ref{t.M-L-piece-HD}.

\subsection{Lower bounds on $HD(M\setminus L)$}

Note that the definition of $X$ in \eqref{e.Cantor-X} implies that $X$ contains the Gauss-Cantor set $K(\{1, 2_2\})$. Thus, Theorem \ref{t.M-L-piece-HD} implies that:

\begin{corollary}\label{c.HD(M-L)>0} One has $HD(M\setminus L)\geq HD(X) \geq HD(K(\{1, 2_2\}))>0$.
\end{corollary}

In Section \ref{a.PT} below, we complete the proof of Theorem \ref{t.A} by employing some classical bounds on Hausdorff dimensions of dynamical Cantor sets discussed in \cite[pp. 68--70]{PT} to obtain the following refinement of the previous corollary:

\begin{proposition}\label{p.HD(M-L)>0} One has $HD(M\setminus L)\geq HD(K(\{1, 2_2\})) > 0.353$.
\end{proposition}

\begin{remark} Of course, this estimate can be improved by computing the value $HD(X)$ using one of the several methods in the literature (e.g., \cite{B}, \cite{H}, \cite{PT}, \cite{JP01}, \cite{JP16}, \cite{FN}). 
\end{remark}

\section{$0.353<HD(K(\{1,2_2\}))<0.35792$}\label{a.PT}

In this section, we revisit pages 68, 69 and 70 of Palis-Takens book \cite{PT} to give some bounds on the Hausdorff dimension of the Gauss-Cantor set $K(\{1, 2_2\})$.

By Lemma \ref{l.0}, the convex hull of $K(\{1,2_2\})$ is the interval $I$ with extremities $[0;\overline{2}]$ and $[0;1,\overline{2}]$. The images $I_1:=\phi_1(I)$ and $I_{22} := \phi_{22}(I)$ of $I$ under the inverse branches
$$\phi_1(x):=\frac{1}{1+x} \quad \textrm{and} \quad \phi_{22}(x) := \frac{1}{2+\frac{1}{2+x}}$$
of the first two iterates of the Gauss map $G(x):=\{1/x\}$ provide the first step of the construction of the Cantor set $K(\{1,2_2\})$. In general, given $n\in\mathbb{N}$, the collection $\mathcal{R}^n$ of intervals of the $n$th step of the construction of $K(\{1,2_2\})$ is given by
$$\mathcal{R}^n:=\{\phi_{x_1}\circ\dots\circ\phi_{x_n}(I): (x_1,\dots, x_n)\in\{1, 22\}^n\}.$$

By definition, $K(\{1, 2_2\})$ is a dynamically defined Cantor set associated to the expanding map $\Psi: I_1\cup I_{22}\to I$ with $\Psi|_{I_1}=G$, $\Psi|_{I_{22}}=G^2$. Following \cite[pp. 68--69]{PT}, given $R\in\mathcal{R}^n$, let
$$\lambda_{n,R}:=\inf\limits_{x\in R}|(\Psi^n)'(x)|, \quad \Lambda_{n,R}:=\sup\limits_{y\in R}|(\Psi^n)'(y)|,$$
and define $\alpha_n\in [0,1]$, $\beta_n\in [0,1]$ by
$$\sum\limits_{R\in\mathcal{R}^n} \left(\frac{1}{\Lambda_{n,R}}\right)^{\alpha_n} = 1 = \sum\limits_{R\in\mathcal{R}^n} \left(\frac{1}{\lambda_{n,R}}\right)^{\beta_n}.$$

It is shown in \cite[pp. 69--70]{PT} that $\alpha_n\leq HD(K(\{1,2_2\}))\leq\beta_n$ for all $n\in\mathbb{N}$.

Therefore, we can estimate on $K(\{1,2_2\})$ by computing $\alpha_n$ and $\beta_n$ for some particular values of $n\in\mathbb{N}$.

In this direction, let us notice that the quantities $\lambda_{n,R}$ and $\Lambda_{n,R}$ can be calculated along the following lines.

Since:
\begin{itemize}
\item $G'(x)=-1/x^2$;
\item the interval $R=\psi_{x_1}\circ\dots\circ\psi_{x_n}(I)\in\mathcal{R}^n$ associated to a string $(x_1,\dots, x_n)\in\{0,1\}^n$ has extremities $[0;x_1,\dots,x_n,\overline{2}]$ and $[0;x_1,\dots,x_n,1,\overline{2}]$, and
\item $(\Psi^n)'|_{R}$ is monotone\footnote{Because $(\Psi^n)|_{R}$ is a M\"obius transformation induced by an integral matrix of determinant $\pm 1$.} on each $R\in\mathcal{R}^n$,
\end{itemize}
we have that
$$\lambda_{n,R} = \min\left\{\prod\limits_{i=1}^{n}\left(\frac{1}{[0;x_i,\dots, x_n, \overline{2}]}\right)^2, \prod\limits_{i=1}^{n}\left(\frac{1}{[0;x_i,\dots, x_n, 1, \overline{2}]}\right)^2 \right\}$$
and
$$\Lambda_{n,R} = \max\left\{\prod\limits_{i=1}^{n}\left(\frac{1}{[0;x_i,\dots, x_n, \overline{2}]}\right)^2, \prod\limits_{i=1}^{n}\left(\frac{1}{[0;x_i,\dots, x_n, 1, \overline{2}]}\right)^2 \right\}.$$

Hence, $\alpha_n$ and $\beta_n$ are the solutions of
$$\sum\limits_{(x_1,\dots,x_n)\in\{1,22\}^n}\left(\min\{[0;x_i,\dots, x_n, \overline{2}], [0;x_i,\dots, x_n, 1, \overline{2}]\}\right)^{2\alpha_n}=1$$
and
$$\sum\limits_{(x_1,\dots,x_n)\in\{1,22\}^n}\left(\max\{[0;x_i,\dots, x_n, \overline{2}], [0;x_i,\dots, x_n, 1, \overline{2}]\}\right)^{2\beta_n}=1.$$

A computer search\footnote{See the Mathematica routine available at `www.impa.br/$\sim$cmateus/files/G(1,22)vPT.nb'.} for the values of $\alpha_{12}$ and $\beta_{12}$ reveals that
$$\alpha_{12} = 0.353465... \quad \textrm{and} \quad \beta_{12} = 0.357917... .$$

In particular, $0.353<\alpha_{12}\leq HD(K(\{1,2_2\}))\leq \beta_{12}<0.35792$, so that the proof of Proposition \ref{p.HD(M-L)>0} and, \emph{a fortiori}, Theorem \ref{t.A} is now complete.

\begin{remark} In general, the approximations $\alpha_n$ and $\beta_n$ given in \cite[pp.68--70]{PT} converge \emph{slowly} to the actual value of the Hausdorff dimension: indeed, as it is explained in \cite[pp.70]{PT}, one has $\beta_n-\alpha_n=O(1/n)$. Hence, it is unlikely that further computations with $\alpha_n$ and $\beta_n$ will lead to the determination of the first ten decimal digits of $HD(K(\{1,2_2\}))$.

On the other hand, a quick implementation\footnote{See the Mathematica routine available at `www.impa.br/$\sim$cmateus/files/G(1,22)vJP.nb'.} of the ``thermodynamical'' algorithm described in Jenkinson-Pollicott \cite{JP01} provided the \emph{heuristic} approximations
\begin{eqnarray*}
& & s_2 = 0.383019..., \quad s_4 = 0.355052...,  \quad 0.35540064 < s_6 < 0.35540065 \\
& & 0.3554004 < s_8 < 0.35554005, \quad 0.355400488 < s_{10} < 0.355400489 \\
& & 0.3553986<s_{12}<0.3553987,
\end{eqnarray*}
for $HD(K(\{1,2_2\}))$. In particular, the super-exponential convergence\footnote{I.e., $|s_n- HD(K(\{1,2_2\}))|=O(\theta^{n^2})$ for some $0<\theta<1$.} of this algorithm \emph{indicates} that $HD(K(\{1,2_2\}))=0.355\dots$. In principle, this heuristics can be made rigorous along the lines of the recent paper \cite{JP16}, but we have not pursued this direction.
\end{remark}

\section{New numbers in $M\setminus L$}\label{a.new-numbers}

Consider the sequences $g, G\in \{1,2\}^{\mathbb{Z}}$ given by
$$g:=\overline{2_4,1_2,2,1},2_5,1,2,1_2,2_3,1,2^*,\overline{1_2,2_3,1,2}$$
and
$$G:=\overline{2,1_2,2,1_2,2,1,2_3,1},2_2,1,2,1_2,2_3,1,2^*,\overline{1_2,2_3,1,2}$$
where the asterisks serve to indicate the zeroth position.

In this section, we show that
\begin{eqnarray*}
c&:=&\lambda_0(G)=[2;\overline{1_2,2_3,1,2}]+[0;1,2_3,1_2,2,1,2_2,\overline{1,2_3,1,2,1_2,2,1_2,2}] \\
&=& \frac{77+\sqrt{18229}}{82}+\frac{17633692-\sqrt{151905}}{24923467}=3.29304447990138\dots
\end{eqnarray*} 
and
\begin{eqnarray*}
\gamma&:=&\lambda_0(g)= [2;\overline{1_2,2_3,1,2}]+[0;1,2_3,1_2,2,1,2_5,\overline{1,2,1_2,2_4}] \\ 
&=& \frac{77+\sqrt{18229}}{82}+\frac{7219908-18\sqrt{82}}{10204619} = 3.29304426427375...
\end{eqnarray*}
are the largest and smallest elements of $(M\setminus L)\cap (b_{\infty}, B_{\infty})$.

\subsection{The largest element of $(M\setminus L)\cap (b_{\infty}, B_{\infty})$} We start the discussions by showing that $c
\in M$:

\begin{lemma}\label{l.c-in-M} One has $c=\lambda_0(G)=m(G)\in M$.
\end{lemma}

\begin{proof} By items (a) and (b) of Lemma \ref{l.2}, we have $\lambda_j(G)<\alpha_{\infty}-10^{-5} < c=\lambda_0(G)$ for all $j\in\mathbb{Z}\setminus\{0\}$ except possibly for
\begin{itemize}
\item $j=-22-12k$, $k\geq 0$,
\item $j=-19-12k$, $k\geq 0$,
\item $j=-16-12k$, $k\geq 0$,
\item $j=-7$,
\item $j=7k$, $k\geq 1$.
\end{itemize}

By item (c) of Lemma \ref{l.2}, we have $\lambda_{-19-12k}(G)<\alpha_{\infty}-10^{-5}<c$ for all $k\geq 0$. By item (d) of Lemma \ref{l.2}, we also have $\lambda_{-22-12k}(G), \lambda_{-16-12k}(G) < \alpha_{\infty}-10^{-5}<c$ for all $k\geq 0$. By item (e) of Lemma \ref{l.2}, we get $\lambda_{-7}(G)<\alpha_{\infty}-10^{-5}<c$.

Moreover, by Lemma \ref{l.0}, we have that
\begin{eqnarray*}
\lambda_{7k}(G) &=& [2;\overline{1_2,2_3,1,2}]+[0;\underbrace{1,2_3,1_2,2,\dots,1,2_3,1_2,2}_{k+1 \textrm{ times }},1,2_2,\overline{1,2_3,1,2,1_2,2,1_2,2}] \\ &<& [2;\overline{1_2,2_3,1,2}]+[0;1,2_3,1_2,2,1,2_2,\overline{1,2_3,1,2,1_2,2,1_2,2}] = \lambda_0(G)
\end{eqnarray*}
for all $k\geq 1$.

In summary, we proved that $\lambda_j(G)<\lambda_0(G)$ for all $j\neq 0$, and, hence, $c=\lambda_0(G)=m(G)\in M$.
\end{proof}

Let us now prove that $m\leq c$ whenever $m\in M\cap(b_{\infty}, B_{\infty})$:
\begin{lemma}\label{l.c-bound}
  If $m\in M\cap(b_{\infty}, B_{\infty})$, then $m\leq c$.
\end{lemma}

\begin{proof} By Proposition \ref{p.M-L-piece-HD2}, an element $m\in M\cap(b_{\infty}, B_{\infty})$ has the form
$$m=\lambda_0(B)=m(B) = [2;\overline{1_2,2_3,1,2}]+[0;1,2_3,1_2,2,1,2_2,\dots].$$
By Lemma \ref{l.0}, we have
$$\lambda_0(B)\leq [2;\overline{1_2,2_3,1,2}]+[0;1,2_3,1_2,2,1,2_2,1,2,\dots].$$
Since $\lambda_0(B)=m<B_{\infty}$, it follows from Lemma \ref{l.Binfty} that
$$\lambda_0(B)\leq [2;\overline{1_2,2_3,1,2}]+[0;1,2_3,1_2,2,1,2_2,1,2_2,\dots].$$
By Lemma \ref{l.0}, we deduce that
$$\lambda_0(B)\leq [2;\overline{1_2,2_3,1,2}]+[0;1,2_3,1_2,2,1,2_2,1,2_3,1,2,1,\dots].$$
Since $\lambda_{-16}(B)\leq m<B_{\infty}<\alpha_{\infty}+10^{-6}$, it follows from items (i), (ii), (iv) and (v) of Lemma \ref{l.1} that
$$\lambda_0(B)\leq [2;\overline{1_2,2_3,1,2}]+[0;1,2_3,1_2,2,1,2_2,1,2_3,1,2,1_2,2,1,\dots].$$
By Lemma \ref{l.0}, we have
$$\lambda_0(B)\leq [2;\overline{1_2,2_3,1,2}]+[0;1,2_3,1_2,2,1,2_2,1,2_3,1,2,1_2,2,1_2,2,1,2,\dots].$$
Since $\lambda_{-22}(B)\leq m<B_{\infty}<\alpha_{\infty}+10^{-6}$, it follows from item (iii) of Lemma \ref{l.1} that
$$\lambda_0(B)\leq [2;\overline{1_2,2_3,1,2}]+[0;1,2_3,1_2,2,1,2_2,1,2_3,1,2,1_2,2,1_2,2,1,2_2,\dots].$$

At this point, we proceed by induction: if we apply repeatedly Lemma \ref{l.0}, items (i), (ii), (iv) and (v) of Lemma \ref{l.1} at the positions $-16-12k$ for $k\geq 1$, and item (iii) of Lemma \ref{l.1} at the positions $-22-12k$ for $k\geq 1$, then we obtain
$$\lambda_0(B)\leq [2;\overline{1_2,2_3,1,2}]+[0;1,2_3,1_2,2,1,2_2,\overline{1,2_3,1,2,1_2,2,1_2,2}] = c.$$
This completes the proof.
\end{proof}

At this point, Proposition \ref{p.new-numbers} is an immediate consequence of Lemmas \ref{l.c-in-M} and \ref{l.c-bound}.

\subsection{The smallest element of $(M\setminus L)\cap (b_{\infty}, B_{\infty})$} Similarly to the previous subsection, we begin our discussion by showing that $\gamma\in M$:

\begin{lemma}\label{l.gamma-in-M} One has $\gamma=\lambda_0(g)=m(g)\in M$.
\end{lemma}

\begin{proof}
From items (a) and (b) of Lemma \ref{l.2}, it follows that $\lambda_j(g)<\alpha_{\infty}-10^{-5}<\gamma$ for all $j\in\mathbb{Z}\setminus\{0\}$ except possibly for
\begin{itemize}
  \item $j=-15-8k$, $k\geq 0$,
  \item $j=-7$
  \item $j=7k$, $k\geq 1$
\end{itemize}
By item (f) of Lemma \ref{l.2}, $\lambda_{-15-8k}(g), \lambda_{-7}(g)<\alpha_{\infty}-10^{-5}<\gamma$ (for $k\geq 0$). Also, by Lemma \ref{l.0}, we have
\begin{eqnarray*}
 \lambda_{7k}(g) &=& [2;\overline{1_2,2_3,1,2}]+[0;\underbrace{1,2_3,1_2,2,\dots,1,2_3,1_2,2}_{k+1 \textrm{ times }},1,2_5,\overline{1,2,1_2,2_4}] \\
   &<& [2;\overline{1_2,2_3,1,2}]+[0;1,2_3,1_2,2,1,2_5,\overline{1,2,1_2,2_4}] = \lambda_0(g)
\end{eqnarray*}
for each $k\geq 1$.

In other terms, we showed that $\lambda_j(g)<\lambda_0(g)$ for all $j\neq 0$, and, \emph{a fortiori}, $\gamma=\lambda_0(g)=m(g)\in M$.
\end{proof}

Let us now establish the fact $m\geq\gamma$ for all $m\in M\cap(b_{\infty},B_{\infty})$:
\begin{lemma}\label{l.gamma-bound}
  If $m\in M\cap(b_{\infty},B_{\infty})$, then $m\geq\gamma$.
\end{lemma}

\begin{proof} By Proposition \ref{p.M-L-piece-HD1}, any $m\in M\cap(b_{\infty},B_{\infty})$ has the form:
$$m=\lambda_0(B) = m(B) = [2;\overline{1_2,2_3,1,2}]+[0; 1,2_3,1_2,2,1,2_2,\dots].$$

We claim there exists a smallest integer $k_0\in\mathbb{N}$ such that $B_{-11-7k_0},B_{-12-7k_0}\neq 2,1$: otherwise, since $m(B)<B_{\infty}<\alpha_{\infty}+10^{-6}$, we could recursively apply Lemma \ref{l.third-restriction} at the positions $n=-7k$ to deduce that $B=\overline{2,1_2,2_3,1}$, and, hence $b_{\infty}=m(\overline{2,1_2,2_3,1})=m(B)$, a contradiction with our assumption $m(B)>b_{\infty}$.

Note that the definition of $k_0$ and Lemma \ref{l.third-restriction} imply that
$$m(B)\geq \lambda_{-7k_0}(B)=[2;\overline{1_2,2_3,1,2}]+[0; 1,2_3,1_2,2,1,2_2,B_{-11-7k_0},B_{-12-7k_0},\dots]$$
with $B_{-11-7k_0},B_{-12-7k_0}\neq 2,1$.

If $B_{-11-7k_0}=1$, then we are done because Lemma \ref{l.0} says that
\begin{eqnarray*}
  m(B)&\geq& \lambda_{-7k_0}(B)=[2;\overline{1_2,2_3,1,2}]+[0; 1,2_3,1_2,2,1,2_2,1,B_{-12-7k_0},\dots] \\
   &>& [2;\overline{1_2,2_3,1,2}]+[0;1,2_3,1_2,2,1,2_5,\overline{1,2,1_2,2_4}] = \gamma.
\end{eqnarray*}

If $B_{-11-7k_0}=2$, then $B_{-11-7k_0},B_{-12-7k_0}\neq 2,1$ forces $B_{-12-7k_0}=2$, and, thus,
$$m(B)\geq \lambda_{-7k_0}(B)=[2;\overline{1_2,2_3,1,2}]+[0; 1,2_3,1_2,2,1,2_4,\dots].$$

By Lemma \ref{l.0}, it follows that
$$m(B)\geq \lambda_{-7k_0}(B)=[2;\overline{1_2,2_3,1,2}]+[0; 1,2_3,1_2,2,1,2_5,1,2,1,\dots].$$

At this point, we recursively apply items (i), (ii) and (iv) of Lemma \ref{l.1} at the positions $j=-15-8k-7k_0$, $k\geq 0$ together with Lemma \ref{l.0} to obtain that
\begin{eqnarray*}
 m(B)&\geq& \lambda_{-7k_0}(B)=[2;\overline{1_2,2_3,1,2}]+[0; 1,2_3,1_2,2,1,2_5,1,2,1,\dots] \\
   &\geq& [2;\overline{1_2,2_3,1,2}]+[0;1,2_3,1_2,2,1,2_5,\overline{1,2,1_2,2_4}] = \gamma.
\end{eqnarray*}

In any case, we proved that $m\geq \gamma$, as desired.
\end{proof}

\appendix

\section{Berstein's interval around $\alpha_{\infty}$}\label{a.Berstein}

In this appendix, we prove that $(b_{\infty},B_{\infty})$ is the largest interval disjoint from $L$ containing $\alpha_{\infty}$. 

\begin{remark}
  The first attempt to describe the largest interval $(b_{\infty}, B_{\infty})$ disjoint from $L$ containing $\alpha_{\infty}$ was made by Berstein \cite{Be73} in 1973: for this reason, we refer to $(b_{\infty}, B_{\infty})$ as Berstein's interval around $\alpha_{\infty}$. As it turns out, his description of $b_{\infty}$ and $B_{\infty}$ in Theorem 1, page 47 of \cite{Be73} is slightly different from ours (perhaps due to some typographical errors). More precisely:
  \begin{itemize}
    \item our value of $b_{\infty}=\ell(\overline{2,1_2,2_3,1}) = 3.2930442439\dots$ is slightly smaller than the value $[2;\overline{1_2,2_3,1,2}]+[0;1,2_3,1_2,2,1,2_5,\overline{1,2,1_2,2_4}] = 3.2930442642\dots$ proposed by Berstein\footnote{Actually, this value proposed by Berstein coincides with the smallest element $\gamma$ of $M\cap (b_{\infty}, B_{\infty})$: see Appendix \ref{a.new-numbers}.};
    \item our value of $B_{\infty} = 3.2930444814\dots$ coincides with the \emph{numerical} value proposed by Berstein, but curiously enough Berstein also claims that $3.2930444814\dots$ equals\footnote{We \emph{guess} that Berstein wanted to write $[2;1,\overline{1,2_3,1,2,1_2,2,1_2,2}]+[0;\overline{1,2_3,1_2,2}] = 3.293044481451\dots$ here, but this quantity is slightly larger than $B_{\infty}=3.293044481438\dots$ anyway.} $[2;1,\overline{1,2_3,1,2,1_2,2,1_2,2}]+[0;1,\overline{2_3,1_2}]$, which is certainly not true (as this last number is $3.29306183\dots$).
  \end{itemize}
\end{remark}

As we pointed out in Remark \ref{r.Berstein}, since Proposition \ref{p.Cusick-Flahive-thm4} ensures that $(b_{\infty},B_{\infty})\cap L=\emptyset$ and $b_{\infty}=\ell(\overline{2,1_2,2_3,1})\in L$, our task is reduced to the following lemma:

\begin{lemma}
  One has $B_{\infty}\in L$.
\end{lemma}

\begin{proof}
  Since $L$ is a closed subset of the real line, it suffices to find a sequence $(P_a)_{a\in\mathbb{N}}$ of finite words in $1$ and $2$ such that 
  $$\lim\limits_{a\to\infty} \ell(\overline{P_a})=B_{\infty}.$$
  
  We claim that the finite words  
  $$P_a:=Q_a R S_a$$
  given by concatenation of the blocks 
  $$Q_a:=\underbrace{2,1_2,2,1_2,2,1,2_3,1,\dots,2,1_2,2,1_2,2,1,2_3,1}_{a \textrm{ times }},$$
  $$R:= 2_2, 1, 2, 1_2, 2_3, 1, 2, 1_2, 2_3, 1, 2^*, 1,$$
  and 
  $$S_a:= \underbrace{1, 2_3, 1, 2, 1_2, 2, 1_2, 2, \dots, 1, 2_3, 1, 2, 1_2, 2, 1_2, 2}_{a \textrm{ times }}, 1, 2_3, 1$$
  satisfy $\lim\limits_{a\to\infty} \ell(\overline{P_a})=B_{\infty}$. 
  
  Indeed, we start by noticing that Lemma \ref{l.0} implies that $B_{\infty}+\frac{1}{2^{12 a-1}}>\lambda_j(\overline{P_a})>B_{\infty}$ whenever the $j$th position of $\overline{P_a}$ corresponds to $2^*$ in a copy of the block $R$: for the sake of convenience, we denote by $\mathcal{C}_a$ the set of such positions. Next, we observe that items (a) and (b) of Lemma \ref{l.2} imply that $\lambda_j(\overline{P_a})<\alpha_{\infty}-10^{-5}$ except possibly when the $j$th position of $\overline{P_a}$ corresponds to $2$ in a copy of $Q_a$, $R$ or $S_a$ whose immediate neighborhood is $1,2,1$. By inspecting the blocks $Q_a$, $R$, $S_a$, we see that if the $j$th position of $\overline{P_a}$ corresponds to $2$ in a copy of $Q_a$, $R$ or $S_a$ whose immediate neighborhood is $1,2,1$, then:
  \begin{itemize}
    \item either $j\in \mathcal{C}_a$ corresponds to $2^*$;
    \item or $j+7\in\mathcal{C}_a$ and its neighborhood in $\overline{P_a}$ is $2_2,1,2_2,1,2,1_2,2_3,1,2,1_2,2_3,1,2,1_2,2$;
    \item or the neighborhood of the $j$th position is $1_2, \widetilde{2}, 1_2$ or $1,2,1_2,\widetilde{2},1,2_2$ or $2_2,1,\widetilde{2},1_2,2,1$ or $1,2_2,1,\widetilde{2},1_2,2$ (where $\widetilde{2}$ indicates the $j$th position).
  \end{itemize}
In the second case, Lemma \ref{l.xii} implies that $\lambda_j(\overline{P_a})<B_{\infty}-10^{-9}$. In the third case, the items (c), (d) and (e) of Lemma \ref{l.2} says that $\lambda_j(\overline{P_a})<\alpha_{\infty}-10^{-5}$. 

It follows from this discussion that 
$$B_{\infty}<\ell(\overline{P_a})=m(\overline{P_a})=\lambda_j(\overline{P_a})<B_{\infty}+\frac{1}{2^{12 a-1}}$$ 
where $j\in\mathcal{C}_a$ corresponds to $2^*$ in a copy of the block $R$. This proves the claim. 
\end{proof}

\section{On Cusick-Flahive sequence $(\alpha_n)_{n\in\mathbb{N}}$}\label{a.CF}

Recall that Theorem 4 in Chapter 3 of Cusick-Flahive book \cite{CF} proves that
$$\alpha_n:=\lambda_0(A_n):= [2; \overline{1_2, 2_3, 1, 2}] + [0; 1, 2_3, 1_2, 2, 1, 2_n, \overline{1, 2, 1_2, 2_3}]\in M\setminus L$$
for all $n\geq 4$. 

In this appendix, we show that the largest element $\alpha_2$ of the sequence $(\alpha_n)_{n\in\mathbb{N}}$ belongs to the Lagrange spectrum: 
\begin{proposition}\label{p.alpha2} One has $\alpha_2=3.2930444886\dots\in L$. In particular, $\alpha_4$ is the largest element of the sequence $(\alpha_n)_{n\in\mathbb{N}}$ belonging to $M\setminus L$. 
\end{proposition} 

During the proof of this proposition, we will need the following two lemmas:

\begin{lemma}\label{l.alpha2-1} Let $B\in(\mathbb{N}^*)^{\mathbb{Z}}$ be a bi-infinite sequence. Suppose that $B$ contains the subsequence $2_3121_22_312^*1_22_3121_22_2$. Then, 
$$\lambda_j(B)<\alpha_2 - 3\times 10^{-8}$$
where $j$ is the position indicated by the asterisk. 
\end{lemma}

\begin{proof} By Remark \ref{r.0}, if $B$ contains $2_3121_22_312^*1_22_3121_22_2$, then 
\begin{eqnarray*}
\lambda_j(B)&<&[2;1_2,2_3,1,2,1_2,2_2] + [0;1,2_3,1_2,2,1,2_3] \\ &=& 
\frac{12230321}{3713986} = 3.2930444541\dots < \alpha_2-3\times10^{-8}.  
\end{eqnarray*} 
\end{proof}

\begin{lemma}\label{l.alpha2-2} Let $B\in(\mathbb{N}^*)^{\mathbb{Z}}$ be a bi-infinite sequence. Suppose that $B$ contains the subsequence $121_22_3121_22_312^*1_22_3121_221_2212_312$. Then, 
$$\lambda_j(B)<\alpha_2-6\times 10^{-9}$$
where $j$ is the position indicated by the asterisk.
\end{lemma}

\begin{proof} By Remark \ref{r.0}, if $B$ contains $121_22_3121_22_312^*1_22_3121_221_2212_312$, then 
\begin{eqnarray*}
\lambda_j(B)&<&[2;1_2,2_3,1,2,1_2,2,1_2,2,1,2_3,1,2] + [0;1,2_3,1_2,2,1,2_3,1_2,2,1] \\ &=& 
\frac{22619524795}{6868879214} = 3.2930444822\dots < \alpha_2-6\times10^{-9}.  
\end{eqnarray*} 
\end{proof}

After these preliminaries, we are ready to show Proposition \ref{p.alpha2}:

\begin{proof}[Proof of Proposition \ref{p.alpha2}] Since $L$ is a closed subset of the real line, it suffices to find a sequence $(T_a)_{a\in\mathbb{N}}$ of finite words in $1$ and $2$ such that 
$$\lim\limits_{a\to\infty}\ell(\overline{T_a}) = \alpha_2.$$ 

We affirm that the finite words 
$$T_a:=U_a V W_a$$
where 
$$U_a:=2,1,2_3,1,2,1_2,2,1_2,2,1,2_3,1_2,2,1,\underbrace{2_3,1_2,2,1, \dots, 2_3,1_2,2,1}_{a \textrm{ times }},$$
$$V:=2_3,1_2,2^{**},1,2_3,1_2,2,1,2_2,1,2,1_2,2_3,1,2^*,$$
and 
$$W_a:=\underbrace{1_2,2_3,1,2,\dots,1_2,2_3,1,2}_{a\textrm{ times }},1_2,2_3,1,2,1_2,2_3,1,2,1_2,2,1_2,2,1,2_3,1,2$$ 
satisfy $\lim\limits_{a\to\infty}\ell(\overline{T_a})=\alpha_2$. 

Indeed, we start by observing that Lemma \ref{l.0} implies that $|\lambda_j(\overline{T_a})-\alpha_2| < \frac{1}{2^{7a}}$ whenever the $j$th position of $\overline{T_a}$ corresponds to $2^{**}$ or $2^*$ in a copy of the block $V$: for the sake of convenience, we denote by $\mathcal{D}_a$ the set of such positions. Now, we note that items (a) and (b) of Lemma \ref{l.2} imply that $\lambda_j(\overline{T_a})<\alpha_{\infty}-10^{-5}$ except possibly when the $j$th position of $\overline{P_a}$ corresponds to $2$ in a copy of $U_a$, $V$ or $W_a$ whose immediate neighborhood is $1,2,1$. By inspecting the blocks $U_a$, $V$, $W_a$, we see that if the $j$th position of $\overline{P_a}$ corresponds to $2$ in a copy of $U_a$, $V$ or $W_a$ whose immediate neighborhood is $1,2,1$, then:
  \begin{itemize}
    \item either $j\in \mathcal{D}_a$ corresponds to $2^{**}$ or $2^*$;
    \item or $j\notin\mathcal{D}_a$ corresponds to $2$ in a copy of $V$ and its neighborhood is $2,1_2,2,1,2_2,1$ or $1,2_2,1,2,1_2,2$; 
    \item or $j$ corresponds to $2$ in a copy of $U_a$ or $W_a$ and its neighborhood is  
    \begin{itemize}
    \item $2_2,1,\widetilde{2},1_2,2,1$ or $1_2, \widetilde{2}, 1_2$ or $1,2,1_2,\widetilde{2},1,2_2$,  
    \item or $2,1,2_3,1,2,1_2,2,1_2,2,1,2_3,1_2,\widetilde{2},1,2_3,1_2,2,1,2_3, 1_2, 2,1$, 
    \item or $1,2,1_2,2_3,1,2,1_2,2_3,1,\widetilde{2},1_2,2_3,1,2,1_2,2,1_2,2,1,2_3,1,2$,
    \item or $2_2,1_2,2,1,2_3,1_2,\widetilde{2},1,2_3,1_2,2,1,2_3$, 
    \item or $2_3,1,2,1_2,2_3,1,\widetilde{2},1_2,2_3,1,2,1_2,2_2$,
    \end{itemize} 
    where $\widetilde{2}$ indicates the $j$th position. 
  \end{itemize}
In the second and third cases, it follows from items (c), (d), (e) of Lemma \ref{l.2} and Lemmas \ref{l.alpha2-1} and \ref{l.alpha2-2} that $\lambda_j(\overline{T_a})<\alpha_2-6\times 10^{-9}$. 

Since $1/2^{7a}<6/10^9$ when $a\geq 4$, our discussion so far implies that 
$$|\ell(\overline{T_a})-\alpha_2|<\frac{1}{2^{7 a}}$$ 
for all $a\geq 4$. This concludes the argument. 
\end{proof}


\begin{thebibliography}{99}

\bibitem{Be73} A.~A.~Berstein, \emph{The connections between the Markov and Lagrange spectra}, Number-theoretic studies in the Markov spectrum and in the structural theory of set addition, pp. 16--49, 121--125. Kalinin. Gos. Univ., Moscow, 1973.

\bibitem{B} R.~Bumby, \emph{Hausdorff dimensions of Cantor sets}, J. Reine Angew. Math. 331 (1982), 192--206.

\bibitem{C} T. Cusick, The connection between the Lagrange and Markoff spectra, Duke Math. J. 42 (1975), 507--517.

\bibitem{CF} T.~Cusick and M.~Flahive, \emph{The Markoff and Lagrange spectra}, Mathematical Surveys and Monographs, 30. American Mathematical Society, Providence, RI, 1989. x+97 pp.

\bibitem{FN} R.~Falk and R.~Nussbaum, $C^m$ \emph{eigenfunctions of Perron-Frobenius operators and a new approach to numerical computation of Hausdorff dimension: applications in} $\mathbb{R}^1$, Preprint (2016) available at arXiv:1612.00870

\bibitem{Fr73} G.~A.~Freiman, \emph{Non-coincidence of the Markov and Lagrange spectra}, Number-theoretic studies in the Markov spectrum and in the structural theory of set addition, pp. 10--15, 121--125. Kalinin. Gos. Univ., Moscow, 1973.

\bibitem{H} D.~Hensley, \emph{Continued fraction Cantor sets, Hausdorff dimension, and functional analysis}, J. Number Theory 40 (1992), no. 3, 336--358.

\bibitem{MM} C.~Matheus and C.~G.~Moreira, \emph{New numbers in $M\setminus L$ beyond $\sqrt{12}$: solution to a conjecture of Cusick}, Preprint (2018) available at https://arxiv.org/abs/1803.01230. 

\bibitem{Mo} C.~G.~Moreira, \emph{Geometric properties of the Markov
and Lagrange spectra}, Ann. of Math. (2) 188 (2018), no. 1, 145--170.

\bibitem{JP01} O.~Jenkinson and M.~Pollicott, \emph{Computing the dimension of dynamically defined sets: $E_2$ and bounded continued fractions}, Ergodic Theory Dynam. Systems 21 (2001), no. 5, 1429--1445.

\bibitem{JP16} O.~Jenkinson and M.~Pollicott, \emph{Rigorous effective bounds on the Hausdorff dimension of continued fraction Cantor sets: a hundred decimal digits for the dimension of $E_2$}, Preprint (2016) available at arXiv:1611.09276.

\bibitem{PT} J.~Palis and F.~Takens, \emph{Hyperbolicity and sensitive chaotic dynamics at homoclinic bifurcations}, Fractal dimensions and infinitely many attractors. Cambridge Studies in Advanced Mathematics, 35. Cambridge University Press, Cambridge, 1993. x+234 pp.





\end{thebibliography}
\end{document}